\documentclass[a4paper,12pt]{article}
\usepackage[top=2.5cm,bottom=2.5cm,left=2.5cm,right=2.5cm]{geometry}
\usepackage{cite, amsmath, amssymb}
\usepackage{amsthm,amssymb,amsmath}
\usepackage{epsfig}
\usepackage{calc,curves,ebezier,epic,eepic,graphicx,multiply,rotating}
\usepackage{color}
\usepackage{amssymb,amsmath,latexsym}
\usepackage{lscape}
\usepackage{amsthm}
\usepackage[usenames,dvipsnames]{pstricks}
\usepackage{epsfig}

\allowdisplaybreaks
\newtheorem{prelem}{{\bf Theorem}}

\newtheorem{theorem}{Theorem}
\newtheorem{corollary}[theorem]{Corollary}

\newtheorem{remark}[theorem]{Remark}

\theoremstyle{definition}

\theoremstyle{definition}
\theoremstyle{remark}

\begin{document}

\baselineskip=0.30in

\begin{center}
{\LARGE \textbf{New upper bounds for the forgotten index among bicyclic
graphs}}

\vspace{5mm}

{\large \textbf{A. Jahanbani$^{a\ast}$, L. Shahbazi$^a$, S.M. Sheikholeslami$^{a}$,\\
R. Rasi$^{a}$, and J. Rodr\'iguez$^{b}$}}

\vspace{5mm}

\baselineskip=0.20in

\emph{$^{a}$Department of Mathematics, Azarbaijan Shahid Madani University%
\\[0pt]
Tabriz, Iran} \\[0pt]
\emph{$^{b}$Departamento de Matem\'{a}ticas, Facultad de Ciencias B\'{a}sicas, Universidad de Antofagasta, Av Angamos 601, Antofagasta, Chile.%
} \\[0pt]
%


\vspace{6mm}

\end{center}

\vspace{6mm}

\baselineskip=0.23in

\noindent \textbf{Abstract }

The forgotten topological index of a graph $G$, denoted by $F(G)$, is
defined as the sum of weights $d(u)^{2}+d(v)^{2}$ over all edges $uv$ of $G$%
, where $d(u)$ denotes the degree of a vertex $u$. In this paper, we give
sharp upper bounds of the F-index (forgotten topological index) over
bicyclic graphs, in terms of the order and maximum degree.
\footnotetext[1]{%
Corresponding author.}

\footnotetext{\textit{E-mail addresses:}
\texttt{Akbar.jahanbani92@gmail.com} (A. Jahanbani),
\texttt{s.m.sheikholeslami@azaruniv.ac.ir} (S.M. Sheikholeslami),
\texttt{l.shahbazi@azaruniv.ac.ir} (L. Shahbazi),
\texttt{r.rasi@azaruniv.edu} (R. Rasi),
\texttt{jonnathan.rodriguez@uantof.cl} (J. Rodr\'iguez)
 }



\baselineskip=0.30in

\vspace{4mm}

\section{Introduction}

A topological index is a numeric quantity associated with a molecular graph
that remains invariant under graph isomorphism and encodes at least one
physical or chemical property of the underlying organic molecule.
Topological indices play an important role in predicting the physical as
well as the chemical properties (boiling point, volatility, stability,
solubility, connectivity, chirality and melting point) of chemical
compounds. For more information we refer to \cite%
{Tod-00,Tod-09} and the references cited therein.

Let $G$ be a simple graph with vertex set $V=V(G)$ and edge set $E=E(G)$.
The integers $n=n(G)=|V(G)|$ and $m=m(G)=|E(G)|$ are the order and the size
of the graph $G$, respectively. If $m=n+1$ then we say that $G$ is a bicyclic
graph. The open neighborhood of vertex $v$ is defined as $N(v)=N_{G}(v)=\{u\in
V(G):uv\in E(G)\}$ and the degree of $v$ is $d_{G}(v)=d_{v}=|N(v)|$. The
maximum degree of a graph $G$ is denoted by $\Delta=\Delta (G)$.

Our main objective in this paper is to investigate the forgotten index, denoted by $F\left(G\right) $ for a graph $G$ and defined as
\begin{equation*}
F=F(G)=\sum_{uv\in E}[d(u)^{2}+d(v)^{2}]=\sum_{v\in V}d(v)^{3}.
\end{equation*}%
\ This topological index was named and first studied by Furtula and Gutman \cite{Fur-15}
in 2015, but it first appeared in 1972 \cite{Gut-72} within
the study of structure-dependency of the total $\pi $-electron energy. 
For recent results in the F-index of graphs we refer to \cite{Abd-17,Che-16,Elu-17,Gao-16,Jav-18,Kha-17,Mil-17}.
In this paper, we established sharp upper bounds for the F-index among bicyclic graphs, in terms of the order and the maximum degree.

\section{Upper  bounds on the  Forgotten topological index of  bicyclic graphs}

In this section, we establish new upper bounds for he  forgotten topological index of of bicyclic graphs.
Now we present some known results that will be needed in this section.

If $n$ is a positive integer, then an integer partition of $n$ is a non-increasing sequence of positive integers $y=(x_{1},x_{2},\ldots ,x_{t})$, such that $n=\sum_{i=1}^{t}x_{i}$. If $\Delta \geq x_{1}\geq x_{2}\geq \ldots \geq x_{t}\geq 1$, then $(x_{1},x_{2},\ldots ,x_{t})$ is called a $\Delta $-partition or an integer partition of $n$ on $N_{\Delta}=\{1,2,\ldots ,\Delta \}$.\newline
A $\Delta $-partition $y=(y_{1},y_{2},\ldots ,y_{t})$ of $n$ is called an integer $\Delta $-dominant sequence if the number $\Delta $ in this partition is as large as possible. In other words, if $n=t\Delta $, then $%
y=(\Delta ,\ldots ,\Delta )$ is the integer $\Delta $-dominant sequence and if $n=t\Delta +b$ where $0<b<\Delta $ then $y=(\Delta ,\ldots ,\Delta ,b)$ is the integer $\Delta $-partition.

Let $B$ be a bicyclic graph of order $n$ and maximum degree $\Delta $. For each $i\in \{1,2,\ldots ,\Delta \}$, let $n_{i}$ denote the number of
vertices of degree $i$. Then
\begin{equation}\label{1}
n_{1}+n_{2}+\cdots +n_{\Delta }=n
\end{equation}%
and
\begin{equation}\label{2}
n_{1}+2n_{2}+\ldots +\Delta n_{\Delta }=2m=2n.
\end{equation}%
Subtracting (\ref{1}) from (\ref{2}), yields
\begin{equation}\label{3}
n_{2}+2n_{3}+\ldots +(\Delta -1)n_{\Delta }=n.
\end{equation}
By (\ref{3}), we obtain the $(\Delta -1)$-partition of $n$ as follows:
\begin{equation}\label{4}
(\ \underbrace{\Delta -1,\ldots ,\Delta -1}_{n_{\Delta }}\ ,\ldots ,\
\underbrace{2,\ldots ,2}_{n_{3}}\ ,\ \underbrace{1,\ldots ,1}_{n_{2}}\ ).
\end{equation}
Next result is an immediate consequence of the above discussion.

\begin{corollary}\label{cor1}
For any bicyclic graph $B$ of order $n$ with maximum
degree $\Delta $, the F-index $F(B)=\sum_{v\in V}d_{v}^{3}$ is maximum if
and only if the $(\Delta -1)$-partition (\ref{4}) is a $(\Delta -1)$%
-dominant sequence of $n$.
\end{corollary}

\begin{remark}
In other words, with regard to the $(\Delta -1)$-dominant sequence of $n$, $%
n_{\Delta }$(number of vertices with degree $\Delta -1$) must be maximum. In
this case, sequence $(n_{1},n_{2},\ldots ,n_{\Delta })$ is called a major sequence for $B$.
\end{remark}

\begin{theorem}
{Let $B$ be a bicyclic graph of order n and maximum degree $ \bigtriangleup $ with $ n\equiv 0 $ $ mod (\bigtriangleup-1 ) $. Then
$$F(B)\le({\Delta^2}+\Delta+2)n+26$$
}
\end{theorem}

\begin{proof}
Without loss of generality, assume that $ n=(\bigtriangleup-1)k $.\\ By equality in (\ref{4}), we have
$${n_\Delta}=\frac{{n+2-({n_2}+2{n_3}+\ldots +(\Delta - 2){n_{\Delta- 1}})}}{{\Delta- 1}}=k-r$$\\
where $$r= \dfrac{n_{2}+2n_{3}+\cdot\cdot\cdot+(\bigtriangleup-2)n_{\bigtriangleup-1}-2}{\bigtriangleup-1}.$$
then $ -1\leq r \leq k-1 $ and $ 1\leq n_{\bigtriangleup} \leq k.$ \\
Thus, consider the following cases.\\
\textbf{Case 0.}  $r=-1$.\\
Then, clearly $ n_{\bigtriangleup}=k+1 $. It follows that
$${n_2}+2{n_3}+\ldots+(\Delta-2){n_{\Delta-1}}+(\Delta-1)(k+1)=(\Delta-1)k+2$$ and so
$${n_2}+2{n_3}+\ldots+(\Delta-2){n_{\Delta-1}}=-(\Delta-1)+2$$ that it is not possible. so $ r=0 $, $ \bigtriangleup>3 .$

\noindent
{\bf Case 1.} $r=0.$\\
Thus, $n_{\bigtriangleup}=k$, $n_{3}=1$, $n_{2}=\ldots =n_{\bigtriangleup-1}=0$. Since $ n_{1}+n_{2}+\ldots +n_{\bigtriangleup}=n$, we conclude that $n_{1}=(\bigtriangleup-2)k-1$. By Corollary \ref{cor1}, we obtain
$$({n_1},{n_2},{n_3},\ldots ,{n_{\Delta  - 1}},{n_\Delta })=((\Delta- 2)k-1,0,1,0, \ldots ,0,k)$$
which is the optimal solution and so $F(B)$ is maximum. Therefore,
\begin{align*}
F(B) &\le {F_{\max }}({B_{n,\Delta }}) = {n_1} + {2^3}{n_2} +  \ldots  + {(\Delta  - 1)^3}{n_{\Delta  - 1}} + {\Delta ^3}{n_\Delta } \\
& = (\Delta  - 2)k - 1 + 3^{3} + {\Delta ^3}(k ) \\
 & = ({\Delta ^3} + \Delta  - 2)k +26\\
 & = ({\Delta ^2} + \Delta  + 2)(\Delta  - 1)k+26 \\
 & = ({\Delta ^2} + \Delta  + 2)n +26.
\end{align*}
 \noindent{\bf Case 2.} $ r=1. $\\
 Since  $ n_{\bigtriangleup}=k-1 $, it follows from (\ref{4}) that  $$n_{2}+2n_{3}+\cdots+(\bigtriangleup-2)n_{\bigtriangleup-1}=(\bigtriangleup-1)+2=(\bigtriangleup-2)+3.$$
First let $ \bigtriangleup>4 $, so,
$n_{4}=1 $, $ n_{\bigtriangleup-1}=1 $, $ n_{2}= \ldots =n_{\bigtriangleup-2}=0 $. Since $ n_{1}+n_{2}+ \ldots +n_{\bigtriangleup}=n $, we conclude that $ n_{1}=(\bigtriangleup-2)k-1 $. By Corollary \ref{cor1},
$$({n_1},{n_2},{n_3},{n_4}\ldots ,{n_{\Delta  - 1}},{n_\Delta }) = ((\Delta  - 2)k - 1,0,0,1,0, \ldots ,0,k-1)$$
which is the optimal solution and so $F(B)$ is maximum. Therefore,
\begin{align*}
F(B) \le {F_{\max }}({B_{n,\Delta }})& = {n_1} + {2^3}{n_2} +  \ldots  + {(\Delta  - 1)^3}{n_{\Delta  - 1}} + {\Delta ^3}{n_\Delta } \\
&= (\Delta  - 2)k - 1 + 4^{3} +(\bigtriangleup-1)^{3}+ {\Delta ^3}(k-1 ) \\
  &= ({\Delta ^3} + \Delta  - 2)k +63-3\bigtriangleup^{2}+3\bigtriangleup-1\\
  &= ({\Delta ^3} + \Delta  - 2)k-3\bigtriangleup^{2}+3\bigtriangleup+62\\
 & = ({\Delta ^2} + \Delta  +2)(\bigtriangleup-1)k-3\bigtriangleup(\bigtriangleup-1)+62\\
 &\leq ({\Delta ^2} + \Delta  +2)n-15(\bigtriangleup-1)+62 \\
 &= ({\Delta ^2} + \Delta  +2)n-15\bigtriangleup+77\\
 &\leq ({\Delta ^2} + \Delta  +2)n+2\\
 &< ({\Delta ^2} + \Delta  +2)n+26.
\end{align*}
 Now, if let $ \bigtriangleup=4 $, by Corollary \ref{cor1}
\begin{center}
  $ (n_{1},n_{2},n_{3},n_{4})=(2k-2,1,2,k-1) $
\end{center}
which is the optimal solution. Therefore,
\begin{align*}
F(B) \le {F_{\max }}({B_{n,\Delta }}) &= {n_1} + {2^3}{n_2} + {3^3}{n_3}+{4^3}{n_1} \\
& = (2k - 2) +8+3^{3}(2)+4^{3}(k-1) \\
 &= (4^{3}+4- 2)k -4 \\
 &= (\bigtriangleup^{3}+\bigtriangleup- 2)k -4 \\
 &= (\bigtriangleup^{2}+\bigtriangleup+ 2)(\bigtriangleup-1)k -4 \\
 &= (\bigtriangleup^{2}+\bigtriangleup+2)n-4\\
&< (\bigtriangleup^{2}+\bigtriangleup+2)n+26.
\end{align*}

 Also, if $ \bigtriangleup=3 $, by Corollary \ref{cor1}
\begin{center}
  $ (n_{1},n_{2},n_{3})=(k-3,4,k-1) $
\end{center}
which is the optimal solution. Thus,
\begin{align*}
F(B) \le {F_{\max }}({B_{n,\Delta }}) &= {n_1} + {2^3}{n_2} + {3^3}{n_3} \\
  &= (k - 3) +8(4)+3^{3}(k-1) \\
 &= (3^{3}+3- 2)k +2 \\
 &= (\bigtriangleup^{3}+\bigtriangleup- 2)k +2 \\
&= (\bigtriangleup^{2}+\bigtriangleup+2)(\bigtriangleup-1)k +2 \\
&= (\bigtriangleup^{2}+\bigtriangleup+ 2)n+2\\
 &< (\bigtriangleup^{2}+\bigtriangleup+2)n+26.
\end{align*}
\noindent {\bf Case 3.} $ 2 \leq r<\bigtriangleup-3.$\\
As above, $$ n_{2}+2n_{3}+ ... +(\bigtriangleup-2)n_{\bigtriangleup-1}=(\bigtriangleup-1)r+2=(\bigtriangleup-2)r+r+2.$$
 since $ r+2<\bigtriangleup-1 $, it follows from Corollary \ref{cor1} that
$$({n_1},{n_2}, \ldots ,n_{r+3},\ldots,{n_{\Delta  - 2}},{n_{\Delta  - 1}},{n_\Delta }) = ((\Delta  - 2)k - 1,0, \ldots ,1,\ldots,0,r,k - r)$$
which is the optimal solution. Thus
\begin{align*}
F(B) \le {F_{\max }}({B_{n,\Delta }}) &= {n_1} + {2^3}{n_2} +  \ldots  + {(\Delta  - 1)^3}{n_{\Delta  - 1}} + {\Delta ^3}{n_\Delta } \\
&= (\Delta  - 2)k - 1 + {(r + 3)^3} + {(\Delta  - 1)^3}r + {\Delta ^3}(k - r) \\
& = ({\Delta ^3} + \Delta- 2)k - 1 + {(r+3)^3} - 3{\Delta ^2}r + 3\Delta r - r \\
 & = ({\Delta ^2} + \Delta  + 2)n + {(r + 3)^3} - 3{\Delta ^2}r + 3\Delta r-r-1 \\
 & = ({\Delta ^2} + \Delta  + 2)n + (r+3)^{3} +r(- 3{\Delta ^2} + 3\Delta-1)-1 \\
&< ({\Delta ^2} + \Delta  + 2)n + \bigtriangleup^{3}+(\bigtriangleup-3)(- 3{\Delta ^2} + 3\Delta -1)-1 \\
& = ({\Delta ^2} + \Delta  + 2)n-2\bigtriangleup^{3}+12\bigtriangleup^{2}-10\bigtriangleup+2\\
 & = ({\Delta ^2} + \Delta  + 2)n-\bigtriangleup^{2}(2\bigtriangleup-12)-10\bigtriangleup+2\\
 & < ({\Delta ^2} + \Delta  + 2)n-25(2\bigtriangleup-12)-50+2\\
 & = ({\Delta ^2} + \Delta  + 2)n-50\bigtriangleup+252\\
& < ({\Delta ^2} + \Delta  + 2)n+2\\
 & < ({\Delta ^2} + \Delta  + 2)n+26.
\end{align*}
Because $5<\bigtriangleup.$\\
\noindent{\bf Case 4.} $ \bigtriangleup-3 \leq r\leq k-1.$
Then $${n_2} + 2{n_3} +  \ldots  + (\Delta  - 2){n_{\Delta  - 1}} = (\Delta  - 2)r + r+2.$$
Thereby, there are non-negative integers t, s such that $ r+2=t(\bigtriangleup-2)+s $ with $ 0\leq s<\bigtriangleup-2 $. Hence \\
$$ n_{2}+2n_{3}+ \ldots+(\bigtriangleup-2)n_{\bigtriangleup-1}=(\bigtriangleup-2)(r+t)+s. $$
If $ 0< s< \bigtriangleup-2 $, then\[({n_1},{n_2}, \ldots ,{n_s},{n_{s + 1}},{n_{s + 2}}, \ldots ,{n_{\Delta  - 2}},{n_{\Delta  - 1}},{n_\Delta }) = ((\Delta  - 2)k - (t + 1),0, \ldots ,0,1,0, \ldots ,0,0,t + r,k - r)\]
which is optimal solution and since $ s <\bigtriangleup-2 $ and $ -r \leq3-\bigtriangleup $ and $ \bigtriangleup > 5 $, we obtain
\begin{align*}
F(B) &\le {n_1} + {2^3}{n_2} +  \ldots  + {(\Delta  - 1)^3}{n_{\Delta  - 1}} + {\Delta ^3}{n_\Delta } \\
& = (\Delta  - 2)k - (t + 1) + {(s + 1)^3} + {(\Delta  - 1)^3}(t + r) + {\Delta ^3}(k - r) \\
& = ({\Delta ^3} + \Delta  - 2)k - t - 1 + {(s + 1)^3} + {\Delta ^3}t - 3{\Delta ^2}t + 3\Delta t - t - 3{\Delta ^2}r + 3\Delta r - r \\
 & = ({\Delta ^3} + \Delta  - 2)k + {(s + 1)^3} - 1 - t( - {\Delta ^3} + 3{\Delta ^2} - 3\Delta  + 2) - r(3{\Delta ^2} + 3\Delta  + 1) \\
 & < ({\Delta ^3} + \Delta  - 2)k + {(\Delta  - 1)^3} - 1 - 1( - {\Delta ^3} + 3{\Delta ^2} - 3\Delta  + 2) + (3 - \Delta )(3{\Delta ^2} + 3\Delta  + 1) \\
 &= ({\Delta ^2} + \Delta  + 2)n - {\Delta ^3} + 6{\Delta ^2} - 4\Delta  - 1 \\
&= ({\Delta ^2} + \Delta  + 2)n - {\Delta ^2}(\Delta  - 6) - 4\Delta  - 1 \\
 & < ({\Delta ^2} + \Delta  + 2)n - 25(\Delta  - 6) - 21 \\
& = ({\Delta ^2} + \Delta  + 2)n - 25\Delta  + 129 \\
 & < ({\Delta ^2} + \Delta  + 2)n + 4 \\
 & < ({\Delta ^2} + \Delta  + 2)n + 26.
\end{align*}
  If $ s=0 $, then the optimal solution is \[({n_1},{n_2}, \ldots ,{n_{\Delta  - 2}},{n_{\Delta  - 1}},{n_\Delta }) = ((\Delta  - 2)k - t,0, \ldots ,0,r + t,k - r).\]
Since $ -r < (4-\bigtriangleup) $, $ \bigtriangleup>3 $, we conclude that
\begin{align*}
F(B) &\le {n_1} + {2^3}{n_2} +  \ldots  + {(\Delta  - 1)^3}{n_{\Delta  - 1}} + {\Delta ^3}{n_\Delta } \\
& = (\Delta -2)k-t+{\Delta ^3}t-3{\Delta^2}t+3\Delta t-t-3{\Delta ^2}r + 3\Delta r-r\\
 & = ({\Delta ^3} + \Delta  - 2)k - t(-{\Delta ^3}+3{\Delta ^2}-3\Delta+2)-r(3{\Delta ^2}-3\Delta+ 1)\\
 &< ({\Delta ^3} + \Delta  - 2)k - 1(-{\Delta^3}+3{\Delta ^2}-3\Delta+2)+(4-\Delta)(3{\Delta^2}-3\Delta+ 1)\\
 & = ({\Delta ^2} + \Delta  + 2)n - 2{\Delta ^3} + 12{\Delta ^2} - 10\Delta + 2 \\
 & = ({\Delta ^2} + \Delta  + 2)n - 2{\Delta ^2}(\Delta  - 6) - 10\Delta  + 2 \\
 & < ({\Delta ^2} + \Delta  + 2)n - 50(\Delta  - 6) - 48 \\
 & = ({\Delta ^2} + \Delta  + 2)n - 50\Delta  + 252 \\
& < ({\Delta ^2} + \Delta  + 2)n - 2 \\
 & < ({\Delta ^2} + \Delta  + 2)n + 26.
\end{align*}
\end{proof}

\begin{theorem}
{Let $B$ be a bicyclic graph of order $n$ and maximum degree $ \bigtriangleup $ with $ n\equiv 1 $ $ mod( \bigtriangleup-1 ) $. Then \\
$$ F(B)\leq(\bigtriangleup^{2}+\bigtriangleup+2)n-(\bigtriangleup^{2}+\bigtriangleup-6). $$
}
\end{theorem}

\begin{proof}
Let $ n=(\bigtriangleup-1)k+1 $. Set $$r=\dfrac{n_{2}+2n_{3}+ ... +(\bigtriangleup-2)n_{\bigtriangleup-1}-3}{\bigtriangleup-1}.$$
By equality in (\textbf{4}), we have $${n_\Delta } = k - (\frac{{{n_2} + 2{n_3} +  \ldots  + (\Delta  - 2){n_{\Delta  - 1}} - 3}}{{\Delta  - 1}}) = k - r.$$\\
Then clearly $ -1\leq r \leq k-1 $ and $ 1\leq n_{\bigtriangleup} \leq k.$ We consider the following cases:\\
\textbf{Case 0.} $r=-1 $.\\
Then clearly $n_{\bigtriangleup}=k+1$. It follow that
$${n_2} + 2{n_3} +  \ldots  + (\Delta  - 2){n_{\Delta  - 1}} + (\Delta  - 1)(k + 1) = (\Delta  - 1)k+3$$ and so
$${n_2} + 2{n_3} +  \ldots  + (\Delta  - 2){n_{\Delta  - 1}}  = -(\Delta  -1) + 3$$ that it is not possible. so $ r=0 $, $ \bigtriangleup>3.$\\
\noindent {\bf Case 1.} $ r=0 $.\\
Since $$r=\dfrac{n_{2}+2n_{3}+ ... +(\bigtriangleup-2)n_{\bigtriangleup-1}-3}{\bigtriangleup-1}=0,$$ it follows that $ n_{4}=1 $, $ n_{2}= ... =n_{\bigtriangleup-1}=0$ and $ n_{\bigtriangleup}=k $. From (2), we have $ n_{1}=(\bigtriangleup-2)k $. Now, by Corollary \ref{cor1}, we have
$$({n_1},{n_2},{n_3},{n_4} \ldots ,{n_{\Delta  - 2}},{n_{\Delta  - 1}},{n_\Delta }) = ((\Delta  - 2)k ,0,0,1,0 \ldots ,0,k )$$ which is the optimal solution. Thus
\begin{align*}
F(B) &\le {F_{\max }}({B_{n,\Delta }}) = {n_1} + {2^3}{n_2} +  \ldots  + {(\Delta  - 1)^3}{n_{\Delta  - 1}} + {\Delta ^3}{n_\Delta } \\
& = (\Delta  - 2)k + 8 + {\Delta ^3}(k) \\
&= ({\Delta ^3} + \Delta  - 2)k + 8 \\
 & = ({\Delta ^2} + \Delta  + 2)(\Delta  - 1)k + 8 \\
 & = ({\Delta ^2} + \Delta  + 2)(n - 1) + 8 \\
 & = ({\Delta ^2} + \Delta  + 2)n - ({\Delta ^2} + \Delta  - 6).
\end{align*}
 \noindent {\bf Case 2.} $ r=1 $\\
 Since $n_{\bigtriangleup}=k-1 $, it follows from (\ref{4}) that \[{n_2} + 2{n_3} +  \ldots  + (\Delta  - 2){n_{\Delta  - 1}} = (\Delta-1)  = (\Delta  - 2) + 4.\]
 First let $ \bigtriangleup>5 $. By Corollary \ref{cor1},
 $$({n_1},{n_2},{n_3},{n_4}, \ldots ,{n_{\Delta  - 2}},{n_{\Delta  - 1}},{n_\Delta }) = ((\Delta  - 2)k-1,0,2,0, \ldots ,0,1,k - 1)$$
which is the optimal solution. Thus
\begin{align*}
F(B) &\le {F_{\max }}({B_{n,\Delta }}) = {n_1} + {2^3}{n_2} +  \ldots  + {(\Delta  - 1)^3}{n_{\Delta  - 1}} + {\Delta ^3}{n_\Delta } \\
 &= (\Delta  - 2)k + 53 + {(\Delta  + 1)^3} + {\Delta ^3}(k - 1) \\
 & = ({\Delta ^3} + \Delta  - 2)k + 52 - 3{\Delta ^2} + 3\Delta  \\
&= ({\Delta ^2} + \Delta  + 2)(n - 1) + 52 - 3{\Delta ^2} + 3\Delta  \\
& = ({\Delta ^2} + \Delta  + 2)n - 4{\Delta ^2} + 2\Delta  + 50 \\
 & = ({\Delta ^2} + \Delta  + 2)n - 4({\Delta ^2} + \Delta  - 6) + 6\Delta+26  \\
 &< ({\Delta ^2} + \Delta  + 2)n - ({\Delta ^2} + \Delta  - 6).
\end{align*}

Now, if let $ \bigtriangleup=5 $, by Corollary \ref{cor1},
$$({n_1},{n_2},{n_3},{n_4},{n_5}) = (3k-1,0,2,1,k - 1)$$
which is the optimal solution. Thus
\begin{align*}
F(B) &\le {F_{\max }}({B_{n,\Delta }}) = {n_1} + {2^3}{n_2} + {3^3}{n_3} + {4^3}{n_4} + {5^3}{n_5} \\
&= 3k-1 + + 54 +64+ {5^3}(k - 1) \\
&= ({5^3} + 5 - 2)k - 8 \\
 & = ({\Delta ^3} + \Delta  - 2)k - 8 \\
&= ({\Delta ^2} + \Delta  + 2)(\Delta  - 1)k - 8 \\
 &= ({\Delta ^2} + \Delta  + 2)(n - 1) - 8 \\
& = ({\Delta ^2} + \Delta  + 2)n - ({\Delta ^2} + \Delta  + 10) \\
& = ({\Delta ^2} + \Delta  + 2)n - ({\Delta ^2} + \Delta  - 6) -16 \\
&< ({\Delta ^2} + \Delta  + 2)n - ({\Delta ^2} + \Delta  - 6).
\end{align*}
Also, if let $ \bigtriangleup=4 $. By Corollary \ref{cor1},
$$({n_1},{n_2},{n_3},{n_4}) = (2k - 2,2,2,k - 1)$$ which is the optimal solution. Thus
\begin{align*}
F(B) &\le {F_{\max }}({B_{n,\Delta }}) = {n_1} + {2^3}{n_2} + {3^3}{n_3} + {4^3}{n_4}\\
& = 2k - 2 + 16 +54 + {4^3}(k - 1) \\
 & = ({4^3} + 4 - 2)k + 4 \\
 & = ({\Delta ^3} + \Delta  - 2)(n - 1) + 4 \\
 & = ({\Delta ^2} + \Delta  + 2)n - ({\Delta ^2} + \Delta  - 2) \\
 &= ({\Delta ^2} + \Delta  + 2)n - ({\Delta ^2} + \Delta  - 6) - 4 \\
 & < ({\Delta ^2} + \Delta  + 2)n - ({\Delta ^2} + \Delta  - 6).
\end{align*}
\noindent {\bf Case 3.} $ 2 $ $ \leq r $ $ < $ $ \bigtriangleup-4 $\\
Since $ 2 $ $ \leq r $ $ < $ $ \bigtriangleup-4 $, thus $ \bigtriangleup> 6 $. In this case:\\
\begin{flushleft}

$$Hypothesis:\left\{ \begin{array}{l}
 n =(\Delta-1)k+1 \\
 2 \le r = \frac{{{n_2} + 2{n_3} +  \ldots  + (\Delta  - 2){n_{\Delta  - 1}} - 3}}{{\Delta  - 1}} < \Delta  - 4 \\
 {n_2} + 2{n_3} +  \ldots  + (\Delta  - 2){n_{\Delta  - 1}} = (\Delta  - 2)r + (r + 3) \\
 r + 1 \le \Delta  - 4 \\
 {n_\Delta } = k - r \\
 {n_{\Delta  - 1}} = r \\
 {n_{r + 4}} = 1 \\
 {n_1} = (\Delta  - 2)k \\
 \end{array} \right.$$
\end{flushleft}

and from Corollary \ref{cor1} we obtain
$$({n_1},{n_2}, \ldots ,{n_{r + 3}},{n_{r + 4}},{n_{r + 5}}, \ldots ,{n_{\Delta  - 2}},{n_{\Delta  - 1}},{n_\Delta }) = ((\Delta  - 2)k,0, \dots ,0,1,0, \ldots ,0,r,k - r)$$ which  which is the optimal optimization. Then
 \begin{align*}
 F(U) &\le {F_{\max }}({U_{n,\Delta }}) = {n_1} + {2^3}{n_2} +  \ldots  + {(\Delta  - 1)^3}{n_{\Delta  - 1}} + {\Delta ^3}{n_\Delta}\\
 & = (\Delta  - 2)k + {(r + 4)^3} + {(\Delta  - 1)^3}r + {\Delta ^3}(k - r) \\
& < ({\Delta ^3} + \Delta  - 2)k + {\bigtriangleup^3} + r( - 3{\Delta ^2} + 3\Delta  - 1) \\
 & = ({\Delta ^2} + \Delta  + 2)(n - 1) + {(r + 2)^3} + r( - 3{\Delta ^2} + 3\Delta  - 1) \\
 & < ({\Delta ^2} + \Delta  + 2)n - {\Delta ^2} - \Delta  - 2 + (\bigtriangleup-4)( - 3{\Delta ^2} + 3\Delta  - 1) \\
&= ({\Delta ^2} + \Delta  + 2)n - 2{\Delta ^3} +14 \Delta^{2}  - 14\bigtriangleup + 4 \\
 & = ({\Delta ^2} + \Delta  + 2)n - (2\bigtriangleup-16)(\bigtriangleup^{2}+\bigtriangleup-6) - 42\bigtriangleup+100\\
& < ({\Delta ^2} + \Delta  + 2)n - (\bigtriangleup^{2}+\bigtriangleup-6).
 \end{align*}
Since $\Delta >6$.\\
{\bf Case 4.} $ \bigtriangleup-4 \leq r k-4. $
Then, there are non-negative integers t, s such that $ r+3=t(\bigtriangleup-2)+s $, $ t\geq 1 $ and $ 0\leq s< \bigtriangleup-2 $.
By substituting in (\ref{4}), we have

$${n_2} + 2{n_3} +  \ldots  + (\Delta  - 2){n_{\Delta  - 1}} = (\Delta  - 2)(r + t) + s.$$

$$ Hypothesis:\left\{ \begin{array}{l}
 n = (\Delta  - 2)k -(t+ 1) \\
 \Delta  - 4 \le r = \frac{{{n_2} + 2{n_3} +  \ldots  + (\Delta  - 2){n_{\Delta  - 1}} - 3}}{{\Delta  - 1}} \leq k - 1 \\
 {n_2} + 2{n_3} +  \ldots  + (\Delta  - 2){n_{\Delta  - 1}} = (\Delta  - 2)(r + t) + s \\
 s + 1 \le \Delta  - 2 \\
 {n_\Delta } = k - r \\
 {n_{\Delta  - 1}} = r + t \\
 {n_{s + 1}} = 1 \\
 {n_1} = (\Delta  - 2)k - (t+1) \\
 \end{array} \right.$$

First let $ 0< s $. From Corollary \ref{cor1}, we have
$$({n_1},{n_2}, \ldots ,{n_{s}},{n_{s+1}},{n_{s + 2}}, \ldots ,{n_{\Delta  - 2}},{n_{\Delta  - 1}},{n_\Delta }) = ((\Delta  - 2)k-(t+1),0, \ldots ,0,1,0, \ldots ,0,0,r+t,k - r)$$
which is the optimal solution. Thus
 \begin{align*}
 F(B) &\le {F_{\max }}({B_{n,\Delta }}) = {n_1} + {2^3}{n_2} + \dots + {(\Delta  - 1)^3}{n_{\Delta  - 1}} + {\Delta ^3}{n_\Delta}\\
 &= (\Delta  - 2)k - (t+1) + {(s + 1)^3} + {(\Delta  - 1)^3}(t + r) + {\Delta ^3}(k - r) \\
& = ({\Delta ^3} + \Delta  - 2)k - 1 + {(s + 1)^3} -t({\Delta^3} + 3{\Delta ^2}-3\Delta +2) - r(3{\Delta ^2}-3\Delta +1) \\
& < ({\Delta ^3} + \Delta  - 2)k - 1 + {(\bigtriangleup-1)^3} -1({\Delta^3} + 3{\Delta ^2}-3\Delta +2) + (4-\bigtriangleup)(3{\Delta ^2}-3\Delta +1) \\
&= ({\Delta ^2} + \Delta  + 2)n - \bigtriangleup^{3}+8\bigtriangleup^{2}-8\bigtriangleup-2\\
& = ({\Delta ^2} + \Delta  + 2)n  +(\bigtriangleup-9)(-\bigtriangleup^{2} -\bigtriangleup+6) - 23\bigtriangleup +52\\
&< ({\Delta ^2} + \Delta  + 2)n  + (\bigtriangleup^{2}+\bigtriangleup -6 ).
 \end{align*}
Now let $ s=0 $, then the optimal solution is
 $$({n_1},{n_2}, \ldots ,{n_{\Delta  - 2}},{n_{\Delta  - 1}},{n_\Delta }) = ((\Delta  - 2)k-t+1,0, \ldots ,0,r+t,k - r).$$ where we have that
 $$Hypothesis:\left\{ \begin{array}{l}
 n = (\Delta  - 1)k + 1 \\
 \Delta  - 4 \le r = \frac{{{n_2} + 2{n_3} +  \ldots  + (\Delta  - 2){n_{\Delta  - 1}} - 3}}{{\Delta  - 1}} < k- 1 \\
 {n_2} + 2{n_3} +  \ldots  + (\Delta  - 2){n_{\Delta  - 1}} = (\Delta  - 2)(r + t)  \\
 {n_\Delta } = k - r \\
 {n_{\Delta  - 1}} = r + t \\
 {n_2} = {n_3} =  \ldots  = {n_{\Delta  - 2}} = 0 \\
 {n_1} = (\Delta  - 2)k - t + 1. \\
 \end{array} \right. $$
Therefore
\begin{align*}
 F(B) &\le {F_{\max }}({B_{n,\Delta }}) = {n_1} + {2^3}{n_2} + \ldots + {(\Delta  - 1)^3}{n_{\Delta  - 1}} + {\Delta ^3}{n_\Delta } \\
 & = (\Delta  - 2)k - t+1 + {(\Delta  - 1)^3}(t + r) + {\Delta ^3}(k - r) \\
& = ({\Delta ^3} + \Delta  - 2)k -t({\Delta^3} + 3{\Delta ^2}-3\Delta +2) +1 - r(3{\Delta ^2}-3\Delta +1) \\
& < ({\Delta ^3} + \Delta  - 2)k -1({\Delta^3} + 3{\Delta ^2}-3\Delta +2) +1+ (4-\bigtriangleup)(3{\Delta ^2}-3\Delta +1) \\
&= ({\Delta ^2} + \Delta  + 2)n - 2\bigtriangleup^{3}+11\bigtriangleup^{2}-11\bigtriangleup+1\\
& =({\Delta ^2}+\Delta+ 2)n  +(2\bigtriangleup-13)(-\bigtriangleup^{2} -\bigtriangleup+6) + 36\bigtriangleup +79\\
&<({\Delta ^2}+\Delta  + 2)n  - (\bigtriangleup^{2}+\bigtriangleup -6 ).
 \end{align*}
\end{proof}

\begin{theorem}
{\em Let $B$ be a bicyclic graph of order $n$ and maximum degree $ \bigtriangleup $ with $ n\equiv p $ $ mod( \bigtriangleup-1 ) $ where $ 2 \leq $ $ p < $$ \bigtriangleup-3 $. Then
 $$ F(B) \leq (\bigtriangleup^{2}+\bigtriangleup+2)-p(\bigtriangleup^{2}+\bigtriangleup+2)+p^{3}+9p^{2}+28p+26. $$
}
\end{theorem}

\begin{proof}
Let $n=(\Delta-1)k+p$. Suppose that
$$r=\frac{n_2+2n_3+\ldots+(\Delta-2)n_{\Delta-1}-p-2}{\Delta-1}.$$
By equality in (\ref{4}), we have
$$n_{\Delta}=k-(\frac{n_2+2n_3+\ldots+(\Delta-2)n_{\Delta-1}-p-2}{\Delta-1})=k-r.$$
Then clearly $-1\leq r \leq k-1$ and $1\leq n_{\Delta}\leq k$. We consider the following cases.
\\
\textbf{Case 0.} $ r=-1 $.
\\Then clearly $ n_{\bigtriangleup}=k+1 $. It follow that
$${n_2} + 2{n_3} +  \ldots  + (\Delta  - 2){n_{\Delta  - 1}} + (\Delta  - 1)(k + 1) = (\Delta  - 1)k+p+2$$ and so
$${n_2} + 2{n_3} +  \ldots  + (\Delta  - 2){n_{\Delta  - 1}}  = -(\Delta  -1) + (p+2)$$ that it is not possible. So $ r=0 $, $ \bigtriangleup-1\geq p+2 $ implies that $ \bigtriangleup \geq p+3. $ We consider $ 2 \leq p \leq \bigtriangleup-3. $
\\
\noindent {\bf Case 1.} $ r=0 $.\\
Then $n_{\Delta}=k$ and we by (4) we have
$$n_2+2n_3+\cdots+
(\Delta-2)n_{\Delta-1}=(\Delta-1)k+p+2-(\bigtriangleup-1)k=p+2.$$
it follows from
Corollary (6) that
$$(n_1,n_2,\ldots,n_{p+3},\ldots,n_{\Delta-1},n_{\Delta})=((\bigtriangleup-2)k+p-1,0,\ldots,0,1,0,\ldots,0,k)$$
which is the optimal solution and so
\begin{align*}
F(B)& \le {F_{\max }}({B_{n,\Delta }}) = {n_1} + {2^3}{n_2} +  \ldots  + {\Delta ^3}{n_\Delta } \\
&= (\Delta  - 2)k + p - 1 + {(p + 3)^3} + {\Delta ^3}(k) \\
 & = ({\Delta ^3} + \Delta  - 2)k + {p^3} + 9{p^2} + 28p+26 \\
 & = ({\Delta ^2} + \Delta  + 2)n - p({\Delta ^2} + \Delta  + 2) + {p^3} + 9{p^2} + 28p+26.
\end{align*}
\noindent {\bf Case 2.}\quad $r=1$.\\
Then $n_{\Delta}=k-1$ and by (\ref{4}) we have
$$ n_{2}+2n_{3}+ \ldots+(\bigtriangleup-2)n_{\bigtriangleup-1}=(\bigtriangleup-1)+p+2=(\bigtriangleup-2)+(p+3).$$
Since $5\leq p +3<\bigtriangleup $, we consider three subcases:\\
{\bf Subcase 2.1}  $p+3=\bigtriangleup-1 $.\\
 Then
$$n_{2}+2n_{3}+ \ldots+(\bigtriangleup-2)n_{\bigtriangleup-1}=2(\bigtriangleup-2)+1.$$ Therefore
$$(n_1,n_2,\ldots,n_{\bigtriangleup-2},n_{\bigtriangleup-1},n_{\Delta})=((\Delta-2)k+p-2,1,0,\ldots,0,2,k-1)$$
which is the optimal solution and $ p=\bigtriangleup-2 $ we have
\begin{align*}
F(B) &\le {F_{\max }}({B_{n,\Delta }}) = {n_1} + {2^3}{n_2} +  \ldots  + {\Delta ^3}{n_\Delta } \\
&= (\Delta  - 2)k + p - 2 + 8 + 2{(\Delta  - 1)^3} + {\Delta ^3}(k - 1) \\
 &= ({\Delta ^3} + \Delta  - 2)k + p + {\Delta ^3} - 6{\Delta ^2} + 6\Delta  +4 \\
 & = ({\Delta ^2} + \Delta  + 2)(n-p) + p({\Delta ^2} -4 \Delta  - 2) + p + {\Delta ^3} - 6{\Delta ^2} + 6\Delta  + 4 \\
 & = ({\Delta ^2} + \Delta  + 2)(n-p) + p({\Delta ^2} -4 \Delta  - 2) + p + (\Delta  - 2)({\Delta ^2} - 4\Delta  - 2) \\
& = ({\Delta ^2} + \Delta  + 2)(n-p) + p({\Delta ^2} -4 \Delta  - 2) + p + (p+2)({\Delta ^2} - 4\Delta  - 2) \\
 & = ({\Delta ^2} + \Delta  + 2)(n-p) + p({\Delta ^2} -4 \Delta  - 2)  +2 {\Delta ^2} - 8\Delta  -4) \\
 & = ({\Delta ^2} + \Delta  + 2)n- p( 5\bigtriangleup +3)+2 {\Delta ^2} - 8\Delta  -4\\
 & < ({\Delta ^2} + \Delta  + 2)n - p({\Delta ^2} + \Delta  + 2) + {p^3} + 9{p^2} + 28p+26,
\end{align*}
 since $ - p( 5\bigtriangleup +3)+2 {\Delta ^2} - 8\Delta  -4 <  {p^3} + 9{p^2} + 28p+26. $\\

{\bf Subcase 2.2}  $ p+3=\bigtriangleup-2 $.
\\Then $n_{2}+2n_{3}+\ldots+(\bigtriangleup-2)n_{\bigtriangleup-1}=2(\bigtriangleup-2).$ Therefore
$$(n_1,n_2,\ldots,n_{\bigtriangleup-2},n_{\bigtriangleup-1},n_{\Delta})=((\Delta-2)k+p-1,0,\ldots,0,2,k-1)$$
which is the optimal solution and $ p=\bigtriangleup-3 $ we have
\begin{align*}
 F(B) &\le {F_{\max }}({B_{n,\Delta }}) = {n_1} + {2^3}{n_2} +  \ldots  + {\Delta ^3}{n_\Delta } \\
& = (\Delta  - 1)k + p - 1 + 2{(\Delta  - 1)^3} + {\Delta ^3}(k - 1) \\
  & = ({\Delta ^3} + \Delta  - 1)k + p+\bigtriangleup^{3}-6\bigtriangleup^{2}+6\bigtriangleup-3\\
 & = ({\Delta ^3} + \Delta  - 1)k + p +(\bigtriangleup-2)( {\Delta ^2} - 4\Delta  -2)-7 \\
& = ({\Delta ^3} + \Delta  - 1)k + p + (p+3)({\Delta ^2} - 4\Delta  -2)-7 \\
 & = ({\Delta ^2} + \Delta  + 2)n - p(5\bigtriangleup+3)+3\bigtriangleup^{2}-12\bigtriangleup-13 \\
& < ({\Delta ^2} + \Delta  + 2)n - p({\Delta ^2} + \Delta  + 2) + {p^3} + 9{p^2} + 28p+26.
\end{align*}
since $$- p(5\bigtriangleup+3)+3\bigtriangleup^{2}-12\bigtriangleup-13 < - p({\Delta ^2} + \Delta  + 2) + {p^3} + 9{p^2} + 28p+26.$$

{\bf Subcase 2.3 }  $p+3 \leq \bigtriangleup-3$.
\\Then
$ n_{2}+2n_{3}+\ldots+n_{\bigtriangleup-1}=(\bigtriangleup-2)+(p+3). $ Therefore
$$(n_1,n_2,\ldots,n_{p+4},\cdots,n_{\bigtriangleup-1},n_{\Delta})=((\Delta-2)k+p-1,0,\ldots,0,1,\cdots,0,1,k-1)$$
which is the optimal solution and $ p\leq\bigtriangleup-4 $ and $ \bigtriangleup\geq 5 $, then we have
\begin{align*}
 F(B)& \le {F_{\max }}({B_{n,\Delta }}) = {n_1} + {2^3}{n_2} +  \ldots  + {\Delta ^3}{n_\Delta } \\
& = (\Delta  - 2)k + p - 1 + {(p + 4)^3}+\Delta^3 -3\bigtriangleup^{2}+3\bigtriangleup-1+\Delta^3(k-1) \\
& = (\bigtriangleup^{3}+\Delta  - 2)k+p+(p+4)^{3}-3\bigtriangleup^{2}+3\bigtriangleup-2\\
&= ({\Delta ^2} + \Delta  + 2)n - p({\Delta ^2} + \Delta  + 2) + {p^3} + 12{p^2} + 49p  - 3{\Delta ^2} + 3\Delta+62  \\
 & < ({\Delta ^2} + \Delta  + 2)n - p({\Delta ^2} + \Delta  + 2) + {p^3} + 9{p^2} + 28p+26,
\end{align*}
 since for $\Delta\ge3$, ${p^3} + 12{p^2} + 49p  - 3{\Delta ^2} + 3\Delta+62 < {p^3} + 9{p^2} + 28p+26 $.

\noindent {\bf Case 3.}\quad $2\le r <\Delta-p-1$.\\
By (\ref{4}), we have
$n_2+2n_3+\ldots+(\Delta-2)n_{\Delta-1}=(\Delta-2)r+(p+r+2)$. Since
$r<\Delta-p-1$, it follows from Corollary \ref{cor1} that
$$(n_1,n_2,\ldots,n_{p+r+3},\ldots,n_{\Delta-2},n_{\Delta-1},n_{\Delta})=((\Delta-2)k+p-1,0,\ldots,0,1,0,\ldots,0,r,k-r)$$
which is the optimal solution. On the other hand, we deduce from $ p \leq
\Delta-3$ and $ r <\Delta-p-1$ and $ \bigtriangleup\geq 6 $.
 Thus
 \begin{align*}
 F(B) &\le {n_1} + {2^3}{n_2} +  \ldots  + {(\Delta  - 1)^3}{n_{\Delta  - 1}} + {\Delta ^3}{n_\Delta } \\
  & = (\Delta  - 2)k + p - 1 + {(p + r + 3)^3} + {(\Delta  - 1)^3}r + {\Delta ^3}(k - r) \\
&= ({\Delta ^3} + \Delta  - 2)k + p - 1 + ({p^3} + {(r + 3)^3} + 3{p^2}(r + 3) + 3p{(r + 3)^2}) - 3{\Delta ^2}r + 3\Delta r - r \\
 & = ({\Delta ^3} + \Delta  - 2)k + {p^3} + 9{p^2} + 28p - 1 + {(r + 3)^3} + 3pr(p + r + 2) - 3{\Delta ^2}r + 3\Delta r - r \\
&<({\Delta ^3} + \Delta  - 2)k + {p^3} + 9{p^2} + 28p + 26 + {r^3} + 9{r^2} + 27r + 3pr(\Delta  + 1) - 3{\Delta ^2}r + 3\Delta r - r \\
 & = ({\Delta ^3} + \Delta  - 2)k + {p^3} + 9{p^2} + 28p + 26 + r(r(r + 9) + 27 + 3p(\Delta  + 1) - 3{\Delta ^2} + 3\Delta  - 1) \\
 &<({\Delta ^3} + \Delta  - 2)k + {p^3} + 9{p^2} + 28p + 26 + r((\Delta  - p - 1)(\Delta  - p + 8) + 26 + 3p(\Delta  + 1) \\&- 3{\Delta ^2} + 3\Delta) \\
 & = ({\Delta ^3} + \Delta  - 2)k + {p^3} + 9{p^2} + 28p + 26 + r( - 2{\Delta ^2} + 10\Delta  + p\Delta  + p(p - 4) + 18) \\
 &<({\Delta ^3} + \Delta  - 2)k + {p^3} + 9{p^2} + 28p + 26 + r( - 2{\Delta ^2} + 10\Delta  + \Delta (\Delta  - 3) \\&+ (\Delta  - 3)(\Delta  - 7) + 18) \\
 & = ({\Delta ^2} + \Delta  + 2)(n - p) + {p^3} + 9{p^2} + 28p + 26 + r( - 3\Delta  + 39) \\
&<({\Delta ^2} + \Delta  + 2)(n - p) + {p^3} + 9{p^2} + 28p + 26.
 \end{align*}

\noindent {\bf Case 4.}\quad $\Delta-p-1 \le r \le k-1.$\\
Let $p+r=t(\Delta-2)+s.$ By substituting in (\ref{4}), we have
$$n_2+2n_3+\ldots+(\Delta-2)n_{\Delta-1}=(\Delta-2)(r+t)+s.$$ If $s=0$
then by Corollary \ref{cor1},
$$(n_1,n_2,\ldots,n_{\Delta-2},n_{\Delta-1},n_{\Delta})=((\Delta-2)k+p-(t+1),0,1,0,\ldots,0,r+t,k-r)$$
which is the optimal solution. Since $\Delta-p \le r+1$ and $p <\Delta-3$,
 and clearly
$$\bigtriangleup^{3}-9\bigtriangleup^{2}+10\bigtriangleup-6 < {p^3} + 9{p^2} + 26p+1< {\Delta ^3} -\Delta -23.$$
Thus
\begin{align*}
F(B) &\le {n_1} + {2^3}{n_2} +  \ldots  + {(\Delta  - 1)^3}{n_{\Delta  - 1}} + {\Delta ^3}{n_\Delta } \\
& = (\Delta  - 2)k + p - (t + 1) + {(3)^3} + {(\Delta  - 1)^3}(t+r) + {\Delta ^3}(k - r) \\
& = ({\Delta ^3} + \Delta  - 2)k + p + 26 + {\Delta ^3}t - 3{\Delta ^2}t + 3\Delta t - 2t - 3{\Delta ^2}r + 3\Delta r - r \\
 & = ({\Delta ^3} + \Delta  - 2)k + p + 26 - t({\Delta ^3} + 3{\Delta ^2} - 3\Delta  + 2) - r(3{\Delta ^2} - 3\Delta  + 1) \\
 &<({\Delta ^3} + \Delta  - 2)k + p + 26 - 1({\Delta ^3} + 3{\Delta ^2} - 3\Delta  + 2) + (p - \Delta  + 1)(3{\Delta ^2} - 3\Delta  + 1) \\
 & = ({\Delta ^3} + \Delta  - 2)k + p + 26 + p(3{\Delta ^2} - 3\Delta  + 1) + 2{\Delta ^3} - 3{\Delta ^2} - \Delta  + 25) \\
 & = ({\Delta ^2} + \Delta  + 2)n - p( - 2{\Delta ^2} + 4\Delta  - 2) - 2{\Delta ^3} + 3{\Delta ^2} - \Delta  + 25 \\
 & < ({\Delta ^2} + \Delta  + 2)n - p({\Delta ^2} + \Delta  + 2) + {p^3} + 9{p^2} + 28p + 26.
\end{align*}
Now let $0<s$. Since $s< \Delta-2$, it follows from Corollary \ref{cor1} that
$$(n_1,n_2,\ldots,n_{s},n_{s+1},n_{s+2},\ldots,n_{\Delta-2},n_{\Delta-1},n_{\Delta})=((\Delta-2)k+p-(t+1),0,\ldots,0,1,0,\ldots,0,0,r+t,k-r)$$
which is the optimal solution. Since $2\leq p<\Delta-3$ and $0<s \leq
\Delta-3$ and clearly $ - p( - 2{\Delta ^2} + 4\Delta  - 2) - {\Delta ^3} + 3{\Delta ^2} - \Delta  - 2 < 2\bigtriangleup^{3}-15\bigtriangleup^{2}+10\bigtriangleup-6 < {p^3} + 9{p^2} + 28p + 26. $

Thus
 \begin{align*}
 F(B) &\le {n_1} + {2^3}{n_2} +  \ldots  + {(\Delta  - 1)^3}{n_{\Delta  - 1}} + {\Delta ^3}{n_\Delta } \\
 & = (\Delta  - 2)k + p - (t + 1) + {(s + 3)^3} + {(\Delta  - 1)^3}(t+r) + {\Delta ^3}(k - r) \\
 & = ({\Delta ^3} + \Delta  - 2)k + p - 1 + {(s + 3)^3} - t( - {\Delta ^3} + 3{\Delta ^2} - 3\Delta  + 2) - r(3{\Delta ^2} - 3\Delta  + 1) \\
 &<({\Delta ^3} + \Delta  - 2)k + p - 1 + {\Delta ^3} - 1( - {\Delta ^3} + 3{\Delta ^2} - 3\Delta  + 2) + (p + 1 - \Delta )(3{\Delta ^2} - 3\Delta  + 1) \\
& = ({\Delta ^2} + \Delta  + 2)n - p( - {\Delta ^2} + \Delta  + 2) + p + p(3{\Delta ^2} + 3\Delta  + 1) - {\Delta ^3} + 3{\Delta ^2} - \Delta  - 2 \\
&= ({\Delta ^2} + \Delta  + 2)n - p( - 2{\Delta ^2} + 4\Delta  - 2) - {\Delta ^3} + 3{\Delta ^2} - \Delta  - 2 \\
 &<({\Delta ^2} + \Delta  + 2)n - p({\Delta ^2} + \Delta  + 2) + {p^3} + 9{p^2} + 28p + 26.
 \end{align*}
This completes the proof.
\end{proof}


\vspace{1cm}
{\bf Funding Information:}
J. Rodr\'{\i}guez was supported by MINEDUC-UA project, code ANT-1899 and Funded by the Initiation Program in Research - Universidad de Antofagasta, INI-19-06.\\

\end{document}